\documentclass[
final
]{dmtcs-episciences}

\usepackage[utf8]{inputenc}
\usepackage{subfigure}

\usepackage[round]{natbib}

\usepackage[english]{babel}
\usepackage{comment}
\usepackage{amsmath,amsfonts,amssymb}
\usepackage{amsthm}
\usepackage{graphicx}
\usepackage{mathrsfs}
\usepackage{soul}

\newtheorem{definition}{Definition}
\newtheorem{theorem}{Theorem}
\newtheorem{lemma}{Lemma}
\newtheorem{corollary}{Corollary}

\newtheorem{property}{Property}
\newtheorem{problem}{Problem}

\newcommand{\OR}[1]{$#1$\textnormal{-OR-PCG}}
\newcommand{\AND}[1]{$#1$\textnormal{-AND-PCG}}
\newcommand{\Int}[1]{$#1$\textnormal{-interval-PCG}}

\newcommand{\ORg}{\textnormal{OR-PCG}}
\newcommand{\ANDg}{\textnormal{AND-PCG}}
\newcommand{\Intg}{\textnormal{multi-interval-PCG}}

\newcommand{\I}{\mathcal{I}}
\newcommand{\J}{\mathcal{J}}

\author[Tiziana Calamoneri et al.]{Tiziana Calamoneri\affiliationmark{1}\thanks{I am  supported by Sapienza University of Rome, grants RM122181612C08BB and RM123188D7F7985D.} \and Manuel Lafond\affiliationmark{2} \\ 
  \and Angelo Monti\affiliationmark{1}\thanks{I am  supported by Sapienza University of Rome, grants RM122181612C08BB and RM123188D7F7985D.}  \and  Blerina Sinaimeri\affiliationmark{3}\thanks{I am supported by MUR PRIN Project EXPAND, grant number 2022TS4Y3N.}}
\title[On Generalizations of PCGs]{On Generalizations of Pairwise Compatibility Graphs}
\affiliation{
  Computer Science Department, Sapienza University of Rome, Italy\\
  Universit\'e de Sherbrooke, Canada\\
  Luiss University, Rome, Italy}
\begin{document}
% This is only used if you are compiling for a volume before vol 25
% \publicationdetails{VOL}{2015}{ISS}{NUM}{SUBM}
% This is the new form of collecting the data, starting with vol 25
\publicationdata{vol. 26:3}{2024}{10}{10.46298/dmtcs.12295}{2023-09-19; 2023-09-19; 2024-04-16}{2024-09-10}
\maketitle
\begin{abstract}
A graph $G$ is a pairwise compatibility graph (PCG) if there exists an edge-weighted  tree and an interval $I$, such that each leaf of the tree is a vertex of the graph, and there is an edge $\{ x, y \}$ in $G$ if and only if the weight of the path in the tree connecting $x$ and $y$ lies within the interval $I$. Originating in phylogenetics, PCGs are closely connected to important graph classes like leaf-powers and multi-threshold graphs, widely applied in bioinformatics, especially in understanding evolutionary processes. 

In this paper we introduce two natural generalizations of the PCG class, namely \OR{k}s and \AND{k}s, which are the classes of graphs that can be expressed as union and intersection, respectively, of $k$ PCGs. These classes can be also described using the concepts of the \emph{covering number} and the \emph{intersection dimension} of a graph in relation to the PCG class.  We investigate how the classes of \ORg\ and \ANDg\  are related to PCGs, \Int{k}s and  other graph classes known in the literature. In particular,  we provide upper bounds on the minimum $k$ for which an arbitrary graph $G$ belongs to \Int{k}, \OR{k}\ or \AND{k}\ classes. For  particular graph classes we improve these general bounds. 

Moreover, we show that, for every integer $k$, there exists a bipartite graph that is not in the \Int{k} class, proving that there is no finite $k$ for which the  \Int{k} class contains all the graphs.  This answers  an open question of Ahmed 
and Rahman from 2017.

Finally, using a Ramsey theory argument, we show that for any $k$, there exists graphs that are not in \AND{k}, and graphs that are not in \OR{k}.
\end{abstract}

\section{Introduction}
Gene orthology is one of the most accurate ways to describe differences and similarities in the composition of genomes from different species. Depending on the mode of descent from their common ancestor, genes can be orthologues (\textit{i.e.} genes that are derived by speciation) or   paralogues (\textit{i.e.}  genes that evolved through duplication) \cite{Fitch2000}. The  delineation of orthologous relationships between genes is  indispensable for the reconstruction of the evolution of species and their genomes \cite{Glover2019}.  Indeed, the connection between phylogenetic  tree and gene orthology plays an important role in many areas of biology \cite{Glover2019} and, at the same time, introduces numerous challenging problems in combinatorics and graph theory.  Indeed, from a graph theoretical point of view, the orthology relation on the set of genes can be represented through a graph whose edges are defined by constraints in the phylogenetic tree representing the evolution of the corresponding species. This has lead to numerous problems on the so-called  \emph{tree-definable} graph classes, \textit{i.e.} graphs that can be defined in terms of a tree. These problems have been widely studied and can be mainly divided into two groups depending whether the tree is 
\textit{(i)} vertex-labeled (leading for example to the class of cographs \cite{Corneil1981,Hellmuth2012,Jung1978,Lafond2014}) or 
\textit{(ii)} edge-labeled (leading for example to the class of leaf power graphs \cite{Eppstein2020,Fellows2008,NRT02}).  
In this paper we focus on the second group and more specifically on the class of {\em pairwise compatibility graphs} (PCGs) \cite{Calamoneri2016,KPM03,Rahman2020}  which is a generalization of the class of leaf power graphs. A graph is a PCG if there exists an edge-weighted tree such that each leaf of the tree is a vertex of the graph, and an edge $\{x,y\}$ is present in the graph if and only if the  weight of the path connecting leaves $x$ and $y$ in the tree lies within a given interval.  This has been extended to $k$-interval PCGs, where $k$ given intervals are allowed~\cite{multiintervalPCG}. PCGs were first introduced in the context of sampling subtrees in a phylogenetic tree \cite{KPM03}, however they  can be also used to model evolution in presence of specific biological events  \cite{Hellmuth2017,Hellmuth2020_fitch2,Hellmuth2021_fitch3,Long2020}.  Indeed, the edge weights can represent the number of biological events that have taken place in the evolution of $x$ and $y$ from their least common ancestor.  
During the last decade, PCGs and their extensions have seen a significant amount of research mainly on determining their relation to other graph classes, see \textit{e.g.} \cite{Calamoneri2016,Rahman2020}. 
For example, PCGs whose underlying tree is a star are equivalent to \emph{double-threshold graphs}~\cite{kobayashi2022linear}.  More generally, $k$-interval PCGs contain \emph{multithreshold graphs}, a new graph class that received considerable attention recently~\cite{jamison2020multithreshold,puleo2020some}.  Indeed, $k$-threshold graphs are $\lceil k/2 \rceil$-interval PCGs whose underlying tree is a star.  There are several open questions regarding the structure of multithreshold graphs, and PCGs can offer insights on these questions.
Let us also mention that PCGs whose allowed interval is $\{1, 2, \ldots, k\}$ are known as \emph{$k$-leaf powers}. A longstanding open problem asked whether $k$-leaf powers could be recognized in polynomial time, which was answered positively very recently~\cite{lafond2022recognizing}.  It is possible that the techniques developed there could be applicable to the problem of recognizing PCGs, but a deeper understanding of PCGs and their extensions is required before this can be answered.

As not all graphs are PCGs \cite{BCMP19,DMR15}, and hence, do not have a tree representation, in this paper we consider the problem of determining the minimum number of trees needed to represent the topology of an arbitrary non-PCG graph. 
To this purpose we introduce two classes, namely \ORg s and \ANDg s: a graph is said to be a \OR{k} (\AND{k}) if it is the union (respectively, intersection) of $k$ PCGs. 

These generalizations have applications in computational biology when predicted evolutionary relationships originate from multiple sources of information.  The idea of considering the OR and AND of a graph class was applied in~\cite{hellmuth2018tree} to study the complexity of recognizing $k$-OR cographs, with the aim to explain orthology graphs that originate from more than one tree.
As for applications in graph theory, notice that \ORg s and \ANDg s are related to the \emph{covering number} and to the \emph{intersection dimension}, respectively, with respect to the PCG class. 
Recall that the covering number of $G$ with respect to a class $\mathcal{A}$  of graphs is the minimum $k$ such that $G$ is the union  of $k$ graphs on vertex set $V(G)$, each of which belongs to $\mathcal{A}$ \cite{Knauer2016};
instead, the intersection dimension of a graph $G$ with respect to a class $\mathcal{A}$ of graphs is the minimum $k$ such that $G$ is the intersection of $k$ graphs on vertex set $V(G)$, each of which belongs to $\mathcal{A}$ \cite{Kratochvil1994}.
 
Both our generalizations are strongly related to the \Intg\ class, which is another generalization of the PCG class introduced in \cite{multiintervalPCG}.  In \Intg s, we are allowed to use more than one interval to define a graph: an edge $\{x,y\}$ is present in a graph in \Int{k}\ if and only if the  weight of the path connecting $x$ and $y$ in the tree lies within at least one of the  $k$ \emph{a priori} defined disjoint intervals. 
It is not hard to see that \ORg\ class is a generalization of \Intg.

\paragraph{Results and organization of the paper.} 

In this paper we investigate the classes of PCG, \Intg, \ORg\ and \ANDg: we study how these classes are related to each other and to other graph classes known in the literature. 

More precisely, in Section~\ref{sec:Preliminaries} we give the formal definitions of PCGs and some of its subclasses and \Intg s; we recall some of the properties of these classes which will be used in the rest of the paper. 
In Section~\ref{sec:definitions_and_or} we provide the formal definition of the two new classes of \ORg s and \ANDg s. 

In Section~\ref{sec:interval} we focus on \Intg s:  it is known that every graph $G(V,E)$ is an \Int{|E|} \cite{multiintervalPCG}, it was an open problem to determine whether there exists a constant $k$ for which the \Int{k}\ class contained every graph;  here we answer this question by showing that, for every integer $k$, there exists a bipartite graph that is not in \Int{k}. 
Furthermore, we improve the result of \cite{multiintervalPCG} concerning the minimum $k$ for which an arbitrary graph is in \Int{k}.

In Sections \ref{def:second-def-OR} and \ref{def:second-def-AND},
we investigate the classes \ORg\ and \ANDg, respectively. We provide upper bounds on the minimum $k$ for which arbitrary graphs are in \OR{k} or in  \AND{k} and we improve these bounds for particular graph classes. 
We then use combinatorial arguments to provide a concrete construction of a small graph that is not in \AND{2}. 
Using more abstract Ramsey-type arguments, we then show that for any $k$, there exist graphs that are not in \OR{k} and graphs not in \AND{k}. In fact, this result can be extended to show that the larger classes of $t$-AND $k$-interval-PCG and $t$-OR $k$-interval-PCG, which are respectively the intersection and union of $t$ graphs that are in $k$-interval-PCG, do not contain all graphs for all fixed $t$ and $k$.

In Section~\ref{sec:conclusion} we conclude proposing several open questions.
We believe that the two generalizations of PCGs introduced in this work not only help in better understanding the PCG class itself, but also pave the way to new and challenging combinatorial problems.

\section{Preliminaries}\label{sec:Preliminaries}

Unless otherwise stated, in this paper we will only consider simple graphs, \textit{i.e.} graphs that contain no loops or multiple edges. Moreover, all the {\em trees} are assumed to be unrooted and with edges weighted by nonnegative real numbers. 
Given a tree $T$, we denote by $Leaves(T)$ its leaf set. Given any two leaves  $u$ and $v$ in $Leaves(T)$, we denote by $P_T(u, v)$ the unique path between $u$ and $v$ in $T$ and by $d_{T}(u,v)$ the sum of the weights of the edges on the  path. 
We call $\mu(T)$ the maximum $d_{T}(u,v)$ over all pairs $u, v$. 
For any set of leaves $L \subseteq Leaves(T)$, we  denote by $T_{L}$ the minimal subtree  of $T$ which contains those leaves. 

A known result for trees exploited in this paper is the following.

\begin{lemma}[\cite{YBR10}]\label{lem:three-leaves-subtree} Let $T$ be 
a tree, and $u$, $ v $ and $w$ be three leaves of $T$ such that $P_T(u,v)$ is the longest path in $T_{\{ u, v, w \}}$. 
Let $x$ be a leaf of $T$ other than $u$, $v$ and $w$. Then, $d_T(x,w) \leq \max\{ d_T (x,u), d_T (x,v)\}$.
\end{lemma}

Let $G=(V, E)$ be a simple, undirected, and not necessarily connected graph with vertex set $V$ and edge set $E$. If we are considering more than a graph, we exploit the notation $V(G)$ and $E(G)$ to mean the vertex and edge sets of $G$, without introducing ambiguity.
The {\em complement graph of $G$}, denoted as $\overline{G}$,  is the graph with vertex set $V$ and edge set $\overline{E}$ consisting of all the non-edges of $G$.
Given a graph class $\mathscr{C}$, its complement $\overline{\mathscr{C}}$ consists of all graphs that are the complement of a graph in $\mathscr{C}$.

Given two graphs $G_1=(V,E_1)$ and $G_2=(V,E_2)$ on the same set of vertices, the graph $G=(V,E_1 \cap E_2)$ is the  {\em intersection graph} of $G_1$ and $G_2$, whereas the graph $G=(V,E_1 \cup E_2)$ is their  {\em union graph}.  In this latter case, we also say that $G_1$ and $G_2$ {\em cover} the edges of $G$.

In the following, we recall some  definitions and results that will be useful in the rest of the paper.

\begin{definition} [\cite{KPM03}]
\label{def.PCG}
Given a tree $T$ and an interval $I$ of nonnegative real numbers, the \emph{pairwise compatibility graph} (PCG) associated to $T$  and $I$, denoted as $PCG(T,I)$, is a graph $G=(V,E)$  whose vertex set $V$ coincides with $Leaves(T)$
and $e=\{u,v\}$, with $u\neq v$,  belongs to $E$ if and only if $d_{T}(u,v) \in I$. \\
A graph $G$ is a PCG if there exists the pair $(T,I)$ such that $G=PCG(T,I)$. 
\end{definition}

%\ml{[ML: is it worth mentioning in the above that $u, v$ should be distinct?  If $0 \in I$, some could think that $G$ can have self-loops.  Same goes with def of LPG.]}

Several  closure properties of the PCG class under some graph operations have been studied and  we use the following result:
\begin{property}[\cite{CMPS13}]\label{prop:closure_pcg}
Let $G$ be  a PCG, then:
\begin{enumerate}
\item any graph  obtained from $G$ by adding a new vertex  with degree $1$ is a PCG (Theorem 10);
\item the  graph  obtained from $G$ by adding a new vertex having the same neighborhood of any given vertex $x$ is a PCG (Theorem 11).%
\end{enumerate}
\end{property}

Furthermore, the next property allows us to work only with natural numbers. 
\begin{property}[\cite{CMPS13}]\label{prop:interi}
If a graph $G$ is a $PCG$ then there exist a tree $T$ whose edges are weighted by natural numbers and $a,b$ in $\mathbb{N}$, with $a \leq b$ such that $G=PCG(T,[a,b])$.
\end{property}

The  PCG class generalizes the well known class of leaf powers graphs.

\begin{definition}[\cite{NRT02}]
Let $T$ be a tree and $d_{max}$ a non negative real number, the {\em leaf power graph} (LPG) associated to $T$ and $d_{max}$, denoted as $LPG(T, d_{max})$, is a graph $G=(V,E)$ whose vertex set coincides with $Leaves(T)$ and $e=\{ u,v \}$ belongs to $E$ if and only if $d_T(u,v) \leq d_{max}$.
A graph $G$ is an LPG if there exists a pair $(T, d_{max})$ such that $G=LPG(T, d_{max})$.
\end{definition}

%ML changed this
%It is defined and studied also the complement of the LPG class, namely the min-leaf power graphs.
The complement of the LPG class, namely the min-leaf power graphs, have also been defined and studied.

\begin{definition}[\cite{CP12}]
Let $T$ be a tree and $d_{min}$ a non negative real number, the {\em min-leaf power graph} (min-LPG) associated to $T$ and $d_{min}$, denoted as $mLPG(T, d_{min})$, is a graph $G=(V,E)$ whose vertex set coincides with $Leaves(T)$ and $e=\{ u,v \}$ belongs to $E$ if and only if $d_T(u,v) \geq d_{min}$.
A graph $G$ is an mLPG if there exists the pair $(T, d_{min})$ such that $G=mLPG(T, d_{min})$.
\end{definition}

Note that we can rephrase the definitions of LPGs and mLPGs as special cases of PCGs whose associated intervals are $[0, d_{max}]$ and $[d_{min}, \mu(T)]$, respectively. 
In the following we will always use this formulation.

Initially it was believed that every graph was a PCG \cite{KPM03}, while it is now well known that not all graphs are PCGs (see {\em e.g.} \cite{BCMP19,DMR15,YBR10}).
Hence, the following super-class of PCGs, namely \Intg s,  has been introduced: 

\begin{definition}[\cite{multiintervalPCG}]\label{def:interval}
Given a tree $T$ and $k \geq 1$ disjoint intervals $I_1, \ldots, I_k$ of nonnegative reals, the \emph{\Int{k}} associated to $T$ and $I_1, \ldots, I_k$, denoted as $k\textnormal{-}PCG(T,I_1,\ldots, I_k)$, is a graph  $G=(V,E)$ whose vertex set $V$ coincides with $Leaves(T)$ and
$e=\{u,v\}$ belongs to $E$ if and only if $d_{T}(u,v) \in I_i$ for some $i$, $1\leq i \leq k$. \\
A graph $G$ is a \Int{k} if there exists a tuple $(T,I_1,\ldots,I_k)$ such that $G=k\textnormal{-}PCG(T,I_1,\ldots, I_k)$.
\end{definition}

\begin{figure}[!ht]
\centering
\includegraphics[width=\textwidth]{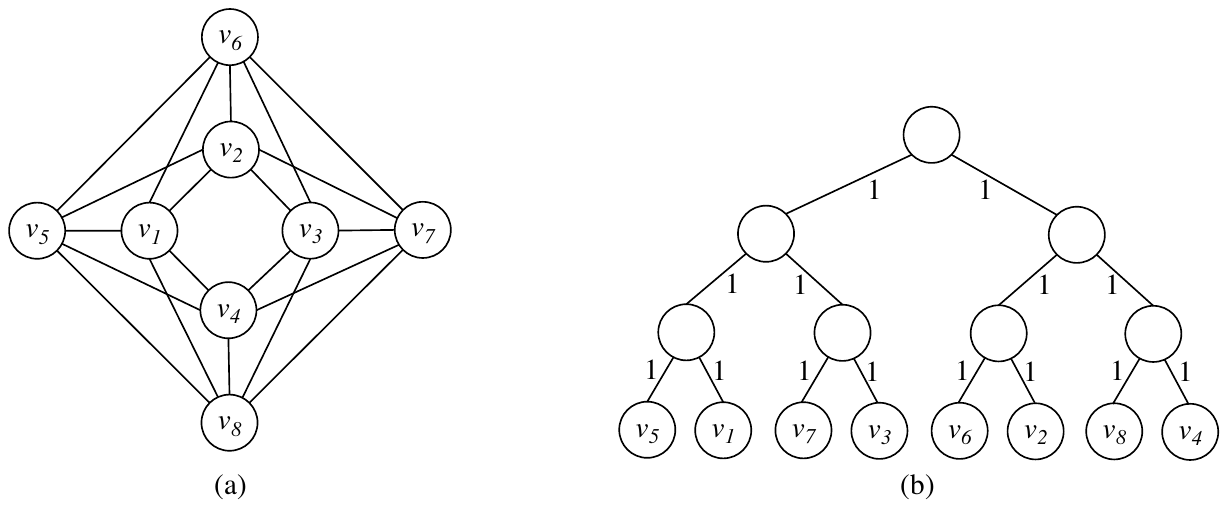}
    \caption{{\bf(a)} A graph G which is not a PCG \cite{DMR15}. {\bf(b)} a tree $T$ such that $G=$ \Int{2}$(T,I_1, I_2)$ where $I_1=[1,3]$ and $I_2=[5,6]$.}\label{y}
\end{figure}

When $k=1$, the \Int{1} class coincides with the PCG class. 
Instead, already when $k=2$ the PCG class is a strict subset of the \Int{2} class; see for example, the graph in Figure~\ref{y}. \\
Other examples of graphs that are \Int{2}s but not PCGs are wheels ({\em i.e.}, one universal vertex connected to all the vertices of a cycle) with at least $9$ vertices  and a restricted subclass of series-parallel graphs \cite{multiintervalPCG}.

\section{Two new generalizations of PCGs}\label{sec:definitions_and_or}

In the generalization of the PCG class to \Intg, given in Definition \ref{def:interval}, the tree remains the same but more than one interval is allowed. It is natural then to consider the further generalization where different trees are also allowed. 

\begin{definition} [\OR{k}]\label{def:or}\label{def:second-def-OR}
Let $T_1,\ldots, T_k$ be $k \geq 1$ trees and $Leaves(T_1)=\ldots=Leaves(T_k)=L$; 
let $I_1, \ldots, I_k$ be $k$ (not necessarily disjoint) intervals of nonnegative real numbers; the \emph{\OR{k}} associated to $T_1, \ldots, T_k$ and $I_1, \ldots, I_k$, denoted as $\OR{k}(T_1,\ldots, T_k,I_1,\ldots, I_k)$, is a graph  $G=(V,E)$ %such that $V=L$ 
whose vertex set $V$ coincides with $L$
and  $e=\{u,v\}$ belongs to $E$ if and only if there exists an $i$, $1\leq i \leq k$, such that $d_{T_i}(u,v) \in I_i$.  
A graph $G$ is a \OR{k} if there exists a tuple $(T_1,\ldots, T_k,I_1,\ldots, I_k)$ such that $G=~\OR{k}(T_1,\ldots, T_k,I_1,\ldots, I_k)$.

Equivalently,  graph $G=(V,E)$ is a \OR{k}$(T_1,\ldots, T_k,I_1,\ldots, I_k)$ if and only if there exist $k$ graphs $G_1, \ldots, G_k$ on the same vertex set $V$ such that for all $i$, $1\leq i \leq k$, it holds that $G_i=PCG(T_i,I_i)$ and $E(G)= \cup_{i=1}^{k}E(G_i)$. 
\end{definition}
\noindent

\begin{figure}[!ht]
\centering
\includegraphics[width=\textwidth]{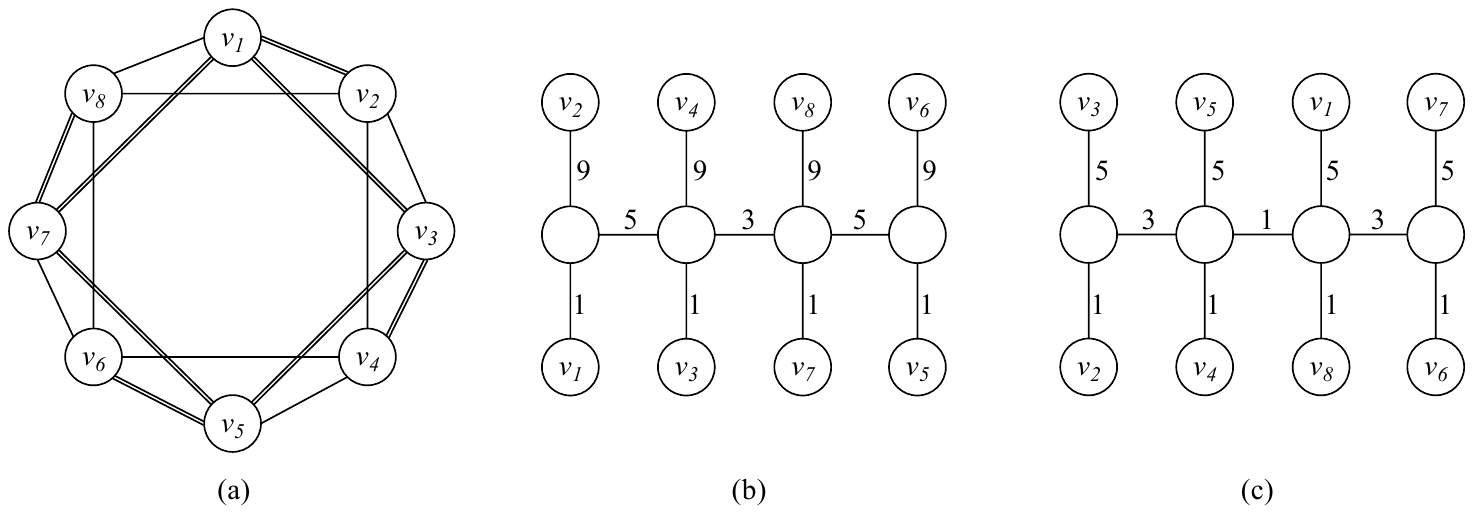}
\vspace*{-5cm}
    \caption{{\bf(a)} A graph $G$ which is not a PCG \cite{BCMP19}. The two graphs on the same set of vertices $G_1$ (induced by the double-lined edges) and $G_2$ (induced by the continuous-lined edges) are in PCG, and we provide:
    {\bf(b)} a tree $T_1$ and an interval $I_1=[6,10]$ such that $G_1=$ PCG$(T_1,I_1)$;
    {\bf(c)} a tree $T_2$ and an interval $I_2=[5,6]$ such that $G_2=$ PCG$(T_2,I_2)$. It clearly holds that $G=$\OR{2}$(T_1, T_2, I_1, I_2)$.}\label{fig:es_or}
\end{figure}
In Figure \ref{fig:es_or} an example of a graph which is a \OR{2}.
Note that in this example, $G$ is the \emph{disjoint} union of two PCGs but, in general, the PCGs that form the graph are allowed to have edges in common.

Since $G$ is a \OR{k} if and only if it can be expressed as the union graph of $k$ PCG subgraphs, we refer to this by  \OR{k}$(G_1,\ldots, G_k)$ when we want to consider the subgraphs instead of the trees and intervals.  It is not hard to see that \ORg\ class is a generalization of \Intg. Indeed, any $G$ in \Int{k} also belongs to \OR{k}, by considering the trees to be identical.
On the other hand, a $k$-OR-PCG $G$ is not necessarily a $k$-interval-PCG, since $G$ could be obtained from the union of PCGs that use different trees.
For example in Figure~\ref{fig:es_or}, the two trees that certify that $G$ is a 2-OR-PCG cannot be used to assess that $G$ is a $2$-interval-PCG.

Requiring that $e=\{u,v\}$ belongs to $G$ if and only if $d_{T_i}(u,v) \in I_i$ for \emph{all} $i$  leads to a different generalization of the PCG class.

\begin{definition} [\AND{k}]\label{def:and}\label{def:second-def-AND}
Let $T_1,\ldots, T_k$ be $k \geq 1$ trees 
and $Leaves(T_1)=\ldots=Leaves(T_k)=L$; let $I_1, \ldots, I_k$ be $k$ (not necessarily disjoint) intervals of nonnegative real numbers; the \emph{\AND{k}} associated to $T_1, \ldots, T_k$ and $I_1, \ldots, I_k$, denoted as $\AND{k}(T_1,\ldots, T_k,I_1,\ldots, I_k)$, is a graph  $G=(V,E)$ 
whose vertex set $V$ coincides with $L$
and $e=\{u,v\}$ belongs to $E$ if and only if for all $i$, $1\leq i \leq k$, $d_{T_i}(u,v) \in I_i$.  
A graph $G$ is a \AND{k} if there exists a tuple $(T_1,\ldots, T_k,I_1,\ldots, I_k)$ such that $G=\AND{k}(T_1,\ldots, T_k,I_1,\ldots, I_k)$.

Equivalently, a graph $G=(V,E)$ is a \AND{k}$(T_1,\ldots, T_k,I_1,\ldots, I_k)$ if and only if there exist $k$ graphs $G_1, \ldots, G_k$ on the same vertex set $V$ such that for all $i$, $1\leq i \leq k$, $G_i=PCG(T_i,I_i)$ and $E(G)= \cap_{i=1}^{k}E(G_i)$. 
\end{definition}

\begin{figure}[!ht]
\centering
\includegraphics[width=\textwidth]{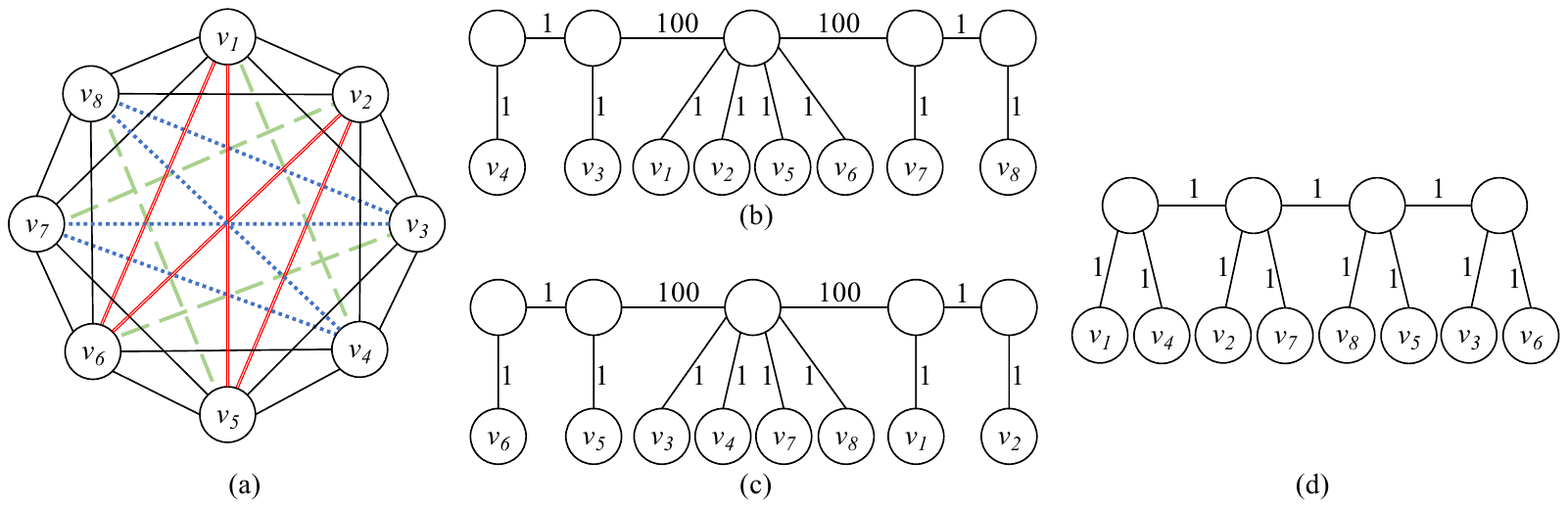}
\vspace*{-5cm}
    \caption{{\bf(a)} A graph $G$ (considering only the black and continuous-lined edges) which is not a PCG \cite{BCMP19}, and three sets of edges: the set $E_1$ of blue and dotted edges, the set $E_2$ of red and double-lined edges, and the set $E_3$ of green and long-dashed edges. 
    The three graphs $G_1=K_8 \setminus E_1$, $G_2=K_8 \setminus E_2$, and $G_3=K_8 \setminus E_3$ are in PCG and we provide:
    {\bf(b)} a tree $T_1$ such that $G_1=$ PCG$(T_1,I_1)$; 
    {\bf(c)} a tree $T_2$ such that $G_2=$ PCG$(T_2,I_2)$;
    {\bf(d)} a tree $T_3$ such that $G_3=$ PCG$(T_3,I_3)$,
    where $I_1=I_2=[2,103]$, and $I_3=[3,5]$. 
    It clearly holds that $G=$~\AND{3}$(T_1, T_2, I_1, I_2, I_3)$.}\label{fig:es_and}
\end{figure}
In Figure \ref{fig:es_and} an example of a graph which is a \AND{3}.

\noindent
Since $G$ is a \AND{k} if and only if it can be expressed as the intersection graph of $k$ PCGs, we refer to this by  \AND{k}$(G_1,\ldots, G_k)$ when we want  to consider the PCGs instead of the trees and intervals. 

Clearly, when $k=1$, the \OR{1} and \AND{1} classes coincide with the PCG class. It is worth to observe that \Int{k}s generalize PCGs using one tree and $k$ intervals while \OR{k}s and \AND{k}s generalize them using $k$ trees and $k$ intervals. 
It is hence natural to consider the intermediate case where PCGs are generalized using $k$ trees and one interval. 
The following result shows that this constraint does not create classes different from \OR{k}s and \AND{k}s.

\begin{theorem}
For any graph $G$, it is a \OR{k}$(T_1,$ $\ldots, T_k$, $I_1,\ldots, I_k)$ (\AND{k}$(T_1$, $\ldots$, $T_k$, $I_1$, $\ldots, I_k)$) if and only if  there exist $k$ trees $T''_1, \ldots, T''_k$ and one interval $I$ of nonnegative real numbers such that  $G=$ \OR{k}$(T''_1,\ldots, T''_k,I,\ldots, I)$ ($G=$ \AND{k}$(T''_1,\ldots, T''_k,I,\ldots, I)$).
\end{theorem}
\begin{proof}
$\Leftarrow$: Trivial, it is sufficient to choose
$I_1=\ldots=I_k=I$.

$\Rightarrow$:
%It is known (\cite{CMPS13}, Lemma 2) that if $G$ is a $PCG$ then there exist a tree $T$ whose edges are weighted by natural numbers and an interval $I$ such that $G=PCG(T,I)$.
Using Property \ref{prop:interi} on $G=$ \OR{k}$(T_1,\ldots, T_k,I_1,\ldots, I_k)$, we can assume that $T_1, \ldots , T_k$ are all weighted by natural numbers. 
Moreover, letting $I_i=[a_i,b_i]$, $1\leq i\leq k$,  we can assume that $b_i > a_i$. 
Indeed, if for some $i$, $1\leq i\leq k$, interval $I_i$ consists of only 1 point (\textit{i.e.}, $a_i=b_i$), 
then we can create a new tree $T^*_i$ by adding 1 to all the edges of $T_i$ incident to a leaf. 
In this way all the distances between two leaves $u_i,v_i$ in $T_i$ are incremented by exactly $2$ in $T^*_i$ and, calling $I^*_i=[a_i+1.5,a_i+2]$, $G_i=PCG(T_i,I_i)=PCG(T^*_i,I^*_i)$.  To this purpose we consider the following cases: (i) $d_{T_i}(u_i, v_i) \in I_i =[a_i]$ then  $d_{T^*_i}(u_i, v_i)=a_i+2 \in I^*_i$; (ii) $d_{T_i}(u_i, v_i) > a_i $ then  $d_{T^*_i}(u_i, v_i)>a_i+2 \not \in I^*_i$; (iii) $d_{T_i}(u_i, v_i) < a_i$ and as we have assumed that the tree $T_i$ is  weighted by natural numbers we have $d_{T_i}(u_i, v_i) < a_i-1$ and thus $d_{T^*_i}(u_i, v_i) < a_i-1+2< a_i +1.5 \not \in I^*_i$. We can hence substitute pair $(T_i,I_i)$ with $(T^*_i,I^*_i)$ in the definition of $G$ as \OR{k}. 

Let $d_i=b_i-a_i >0$ and denote  $D=\Pi_{i=1}^k d_i$. For every $i$, $1\leq i\leq k$, we define
$T'_i$ as obtained from $T_i$ by multiplying the weight of each edge by $\frac{D}{d_i}$ and $I'_i=[a'_i,b'_i]=[\frac{D}{d_i}a_i, \frac{D}{d_i}b_i]$.
(Note that all intervals $I'_1, \ldots, I'_k$ have the same length $D$ and thus for all $i$, $b'_i=a'_i + D$.)
It is easy to see that $G=$\OR{k}$(T'_1,\ldots, T'_k,I'_1,\ldots, I'_k)$.

We now perform a further modification of the trees in order to obtain coinciding intervals.
To this aim, let $A=\max_{1 \leq i \leq k}{\{a'_i\}}$ and, for any $i$, $1\leq i\leq k$, construct tree  $T''_i$ from tree $T'_i$ by increasing the weight of each edge incident to a leaf by $\frac{A-a'_i}{2}$. 
So, for any two leaves $u,v$ we have $d_{T''_i}(u,v)=d_{T'_i}(u,v)+A-a'_i$ and thus $PCG(T'_i,I'_i)=PCG(T''_i,I)$ where $I=[A,A+D]$. 
It follows then  $G=$\OR{k}$(T''_1, \ldots, T''_k, I, \ldots, I)$. 

Using the same arguments, we can prove that the claim holds also for the \AND{k} class.
\end{proof}

%%%%%%%%%%%%%%%%%%%%%%%%%%%%%%%%%%%%%%%%%
\section{On multi-interval-PCGs}\label{sec:interval}
In \cite{multiintervalPCG} it is proved that every graph with $m$ edges is an \Int{m}, thereby deducing that every graph is a \Int{k} for some $k$.
It was then left as an open question to determine the minimum $k$ for which the \Int{k} class contains {\em every} graph. 
It is not known even if this hypothetical $k$ could be 2 as there is no graph known to be outside \Int{2}s. Here, not only we provide the first example of such a graph but we also show that, for any $k\geq 2$ there exists a bipartite graph that requires at least $k$ intervals (and hence is not a \Int{(k-1)}).
This means that there is no constant $k$ for which the \Int{k} class contains all the graphs. 

For each $k\geq 2$ we define bipartite graph $G_k=(V_k,E_k)$ with $V_k=\{x,y\}\cup \{u_i,v_i | 1\leq i\leq 2k-2\}$ and $E_k=
\{ x,y \}\cup\{ \{x,u_{2j-1} \}, \{x,v_{2j-1}\}
| 1 \leq j\leq k-1\}$. 
$G_k$ has $4k-2$ vertices, $2k-1$ edges and $2k-1$ connected components. 
In Figure~\ref{fig:g4}, $G_3$ is depicted.

\begin{figure}[ht!]
\centering
\includegraphics[scale=0.6]{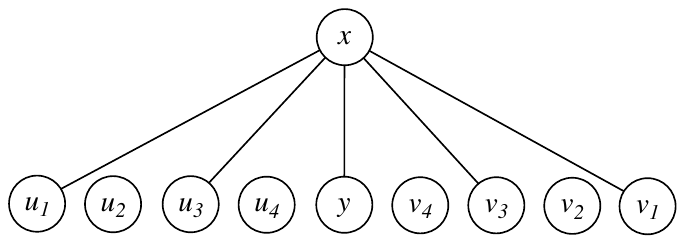}
    \caption{Graph $G_3$.}\label{fig:g4}
\end{figure}

For the sake of simplicity, in what follows in this section we will write $[u_i,v_i]$ to mean the set of vertices $\{u_i, u_i+1 \ldots u_{2k-2}, y, v_{2k-2}, v_{2k-3}\ldots v_i\}$. 

\begin{lemma}\label{secondo} Let $G$ be a graph that contains $G_k$, $k \geq 2$, as an induced subgraph. Let $t$ be the minimum integer for which $G=t\textnormal{-}PCG(T,I_1,\ldots, I_t)$. 
If for each $i$, $1\leq i\leq 2k-2$, path $P_T(u_i,v_i)$ is the longest path in $T_{[u_i,v_i]}$, then  $t\geq k$.
 \end{lemma}
 
\begin{proof}
Let $G=t\textnormal{-}PCG(T,I_1,\ldots, I_t)$. For any $i$, $1\leq i\leq 2k-3$, consider the subtrees $T_{\{u_i,u_{i+1},v_i \} }$ and $T_{\{ u_i,v_{i+1},v_i\} }$. From the hypotheses we have that $P_T(u_i,v_i)$ is the longest path in both the subtrees and hence, using Lemma~\ref{lem:three-leaves-subtree}, we have: 

$$d_T(x, u_{i+1}) \leq \max \{ d_T(x, u_i), d_T(x, v_i)\} \mbox{ and } d_T(x, v_{i+1}) \leq \max \{ d_T(x, u_i), d_T(x, v_i)\},$$ 
from which:

\begin{equation}\label{eq:inequalities}
    \max\{d_T(x,u_i),d_T(x,v_i)\} 
\geq \max\{d_T(x,u_{i+1}),d_T(x,v_{i+1})\}\mbox{ for all $i$}. 
\end{equation}

Moreover, consider the subtree $T_{[u_{2k-2}, v_{2k-2}]}$ and the leaves $u_{2k-2}, v_{2k-2}$ and $y$. We apply Lemma~\ref{lem:three-leaves-subtree} obtaining 

\begin{equation}\label{eq:inequalities2}
\max\{d_T(x,u_{2k-2}),d_T(x,v_{2k-2})\}\geq d_T(x,y).
\end{equation}

Combining (\ref{eq:inequalities}) and (\ref{eq:inequalities2})  we obtain the following chain of inequalities:

{\footnotesize
\begin{equation}\label{eq:inequalities3}
\max\{d_T(x,u_1),d_T(x,v_1)\}\geq \max\{d_T(x,u_2),d_T(x,v_2)\}\geq \ldots \geq  \max\{d_T(x,u_{2k-2}),d_T(x,v_{2k-2})\}\geq d_T(x,y).
\end{equation}
}
 
By the definition of $G_k$, edges $\{ x,u_i \}$ and $\{ x,v_i \}$ belong to $E_k$ if and only if $i$ is odd. 
%da vedere l'ultimo elemento dxy
Hence, in (\ref{eq:inequalities3}), denoting $\max\{d_T(x,u_i),d_T(x,v_i)\}$ by $d_i$, value $d_i$ belongs to some interval $I \in \{I_1, \ldots, I_t\}$  for odd $i$, while it does not belong to any interval  for even $i$, as they correspond to pairs not connected by an edge in $G_k$. 
So, if we consider any odd index $j$, inequalities $d_j \geq d_{j+1} \geq d_{j+2}$ imply that $d_j$ and $d_{j+2}$ cannot belong to the same interval.
Generalizing to the complete chain of inequalities (\ref{eq:inequalities3}) (where the last $d_{j+2}$ coincides with $d_T(x,y)$), we need at least $k$ distinct intervals, hence $t \geq k$.
\end{proof}

We now define the bipartite graph $H_k=(A_k, B_k, E_k)$, $k \geq 2$, as follows:  \\
- vertex sets $A_k$ and $B_k$ are such that $|A_k|=4k-3$ and $|B_k|={{4k-3}\choose{2k-1}}$; \\ 
- edge set $E_k$ is defined in such a way that each vertex in $B_k$ has exactly $2k-1$ neighbors in $A_k$ and no two vertices in $B_k$ have the same neighborhood.  Notice that this is possible due to the size of $B_k$.

\begin{theorem}\label{theo:not-interval}
Graph $H_k=(A_k, B_k, E_k)$  is not a \Int{(k-1)}.
\end{theorem}
\begin{proof}
Let $t$ be the minimum integer for which  $H_k$ is a \Int{t} and let $T$ be the corresponding tree.  Consider the tree $T$ and its subtree $T_{A_k}$, induced by the vertices in $A_k$. It is not difficult to see that it is possible to order the vertices in $A_k$ in the form $u_1, u_2 \ldots u_{2k-2}, y, v_{2k-2}, v_{2k-3}\ldots v_2, v_1$ such that for each $i$, $1\leq i\leq 2k-2$, the path $P_T(u_i,v_i)$ is the longest path in $T_{[u_i,v_i]}$. Indeed, we can start by taking a path of maximum length in $T_{A_k}$ and this will identify the pair of vertices $u_1, v_1$. Subsequently, we can proceed with the same process to discover the pair related to $u_2$ and $v_2,$ and continue this process as needed.

Notice that for any $x \in B_k$ the subgraph induced by $A_k \cup \{ x \}$ is isomorphic to $G_k$. Moreover, for $x\neq x'$ in $B_k$ the corresponding induced subgraphs are different. Hence, we have ${{4k-3}\choose{2k-1}}$ subgraphs isomorphic to $G_k$.  Hence there must exists a vertex $x$ in $B_k$ that is connected to exactly the subset of $2k-2$ vertices in $A_k$ corresponding to 
$\{y,u_{2j-1},v_{2j-1}| 1 \leq j\leq k-1\}$.
%such that the subgraph induced by $U_k \cup \{ x \}$ is $G_k$ and we have  that 
% $P_T(u_i,v_i)$ is the longest path in $T_{[u_i, v_i]}$, for each $i$, $1\leq i\leq 2k-2$. T
Thus by Lemma~\ref{secondo} it must be $t \geq k$.
\end{proof}

\noindent
From the previous theorem we have the following corollary.

\begin{corollary}
There exists no constant $k$ for which the \Int{k} class  contains all graphs.
\end{corollary}

It is worth to mention that, for $k=2$, $H_2$ corresponds to the bipartite graph on $15$ vertices introduced in \cite{YBR10} as the first graph proved to be outside the PCG class. Moreover, the graph $H_3$, consisting of $135$ vertices, is the first graph that is proved not to be a $\Int{2}$.

It remains an interesting open problem finding the smallest graph that is not a $\Int{2}$.

We conclude this section with an improvement of the result in \cite{multiintervalPCG} stating that every $m$ edge graph is an \Int{m}.

\begin{lemma}\label{lem:complement_interval}
Let $G$ be a \Int{k}; then its complement $\overline{G}$ is a \Int{(k+1)}.
\end{lemma}
\begin{proof}
Let $G=$ \Int{k}$(T,I_1,\ldots,I_k)$ and, by definition, the intervals $I_1, \ldots I_k$ are all disjoint; hence, if $I_i=[a_i,b_i]$, $1 \leq i \leq k-1$, then $a_i \leq b_i < a_{i+1}$.
From Property \ref{prop:interi} we can assume $a_i, b_i$, $1 \leq i \leq k$ are all integer values and $a_1 > 0$.
%Let $\mu$ be the maximum distance between any two leaves in $T$.
Consider the $k+1$ intervals $I'_1=[0,a_1-1], I'_2=[b_1+1, a_2-1], \ldots, I'_{k}=[b_{k-1}+1,a_k-1], I'_{k+1}=[b_k+1, \mu(T) +1]$; it is easy to see that $\overline{G}=$\Int{(k+1)}$(T,I'_1,\ldots,I'_{k+1})$.
\end{proof}

\begin{theorem}\label{theo:interval-general}
Let $G$ be a graph with $m$ edges and $n$ vertices and let $t=\min
\{m, \lceil \frac{n^2}{4}- \frac{n}{4}+\frac{1}{2} \rceil \}$  then $G$ is a \Int{t}.
\end{theorem}
\begin{proof}
From \cite{multiintervalPCG}, $G$ is an \Int{m}   and $\overline{G}$ is an \Int{(\frac{n(n-1)}{2}-m)}.  From Lemma~\ref{lem:complement_interval} we have that $G$ is also an \Int{(\frac{n(n-1)}{2}-m+1)}. Hence, denoting $t=\min \{ m, \frac{n(n-1)}{2}-m+1\}$, $G$ is a \Int{t}.
Since if $m \leq \frac{n(n-1)}{2}-m+1$ then $m \leq \frac{n^2}{4}- \frac{n}{4}+\frac{1}{2}$, the result follows.
\end{proof}

Finally notice that, from Lemma~\ref{lem:complement_interval} when $k=1$, we have the following corollary.

\begin{corollary}\label{cor: pcg_compl}
$PCG \cup \overline{PCG} \subseteq ~$\Int{2}.
\end{corollary}

This implies that if a graph $G$ does not belong to the $\Int{2}$ class, then neither $G$ nor $\overline{G}$ is a PCG.

%%%%%%%%%%%%%%%%%%%%%%%%%%%%%%%%%%%%
\section{On OR-PCGs}\label{sec:or}

Recall that \Intg s are trivially \ORg s, and thus  Theorem~\ref{theo:interval-general} holds for \ORg s, too.  
In this section, we improve this result. 

We preliminarily recall a couple of definitions.
A graph $G$ is {\em edge covered} by graphs from a certain class $\mathcal{C}$ if it is possible to select some graphs from $\mathcal{C}$ such that all the edges of $G$ belong to some of the selected graphs, while all the non-edges do not.

Given a graph $G$, its {\em arboricity} is the minimum number of spanning forests needed to cover all the edges of $G$.

Notice that by Definition \ref{def:or},  a graph is in \OR{k} if and only if its  edges can be covered by $k$ PCGs. Thus, we can exploit the wide literature in graph edge covering ({\em e.g.} see the survey paper \cite{Knauer2016}). 

In particular, since forests are PCGs, the following lemma holds:\\
\begin{lemma}
\label{lemma:arboricity}
A graph with arboricity $a$ is in \OR{a}.   
\end{lemma}
% a graph with arboricity $a$ is 
%\st{({\em i.e.} the minimum number of spanning forests needed to cover all the edges of the graph is $a$) }

%We use this result in the proof of the next theorem.  

First, notice that any graph with at most 7 vertices is a PCG \cite{Calamoneri2012} and hence is a \OR{1}, so we assume $n \geq 8$.

\begin{theorem}\label{theo:or_general}
Any graph $G$ with $n\geq 8 $ vertices and maximum degree $\Delta$ is a \OR{\min\{\lceil \frac{3 \Delta +2}{5} \rceil, \lceil\frac{n-7}{3}\rceil +1 \}}.
\end{theorem}

\begin{proof}
It is known that if $G$ has maximum degree $\Delta$, then its edges can be covered by $\lceil \frac{3 \Delta +2}{5} \rceil$ forests of paths \cite{G86}.
Since forests of paths are trivially PCGs, in view of Definition \ref{def:second-def-OR}, $G$ is a \OR{\lceil \frac{3 \Delta +2}{5} \rceil}.

We now show that $G$ is also a \OR{\bigl(\lceil\frac{n-7}{3}\rceil +1\bigr)}. 
We  construct an edge cover with $\lceil\frac{n-7}{3}\rceil +1$ PCGs for $G$ and the result follows again from Definition \ref{def:second-def-OR}.
Consider any $3 \lceil \frac{n-7}{3}\rceil$ vertices of $G$ and partition them into $\lceil \frac{n-7}{3}\rceil$ triples $x_i, y_i$ and $z_i$.
For any $1\leq i \leq \lceil \frac{n-7}{3}\rceil$,  let $G_i=(V_i, E_i)$ be the subgraph  induced by the edges of $G$ incident to $x_i, y_i$ or $z_i$.  Let $G'$ be the subgraph of $G$ induced by the remaining (at most 7) vertices.
$G'$ is a PCG as it has at most 7 vertices, so the proof is concluded by showing that the $\lceil\frac{n-7}{3}\rceil$ graphs $G_i$ are PCGs.

Observe that the structure of every $G_i$ is as shown in Figure~\ref{fig:gadget}: besides vertices $x_i, y_i, z_i$, it contains:
\begin{itemize}
\item
three vertex sets $P^i_{x}$, $P^i_y$ and $P^i_z$ with all vertices adjacent only to $x_i$, only to $y_i$ or only to $z_i$, respectively; note that the vertices in each of these sets have degree 1;
\item 
three vertex sets, $S^i_{xy}$, $S^i_{xz}$ and $S^i_{yz}$ with all vertices adjacent to $x_i$ and $y_i$, $x_i$ and $z_i$, $y_i$ and $z_i$, respectively; note that the vertices in each of these sets have degree 2 and  have the same neighborhood;
\item
vertex set $S^i_{xyz}$ with all vertices adjacent to $x_i$, $y_i$ and $z_i$; note that the vertices in this set have degree 3 and have the same neighborhood.
\end{itemize}
Each one of all these sets may be empty and edges $\{ x_i, y_i\}$, $\{ x_i, z_i\}$ and $\{ y_i, z_i\}$ may be present or not.

\begin{figure}[ht!]
\centering
  \includegraphics[scale=0.7]{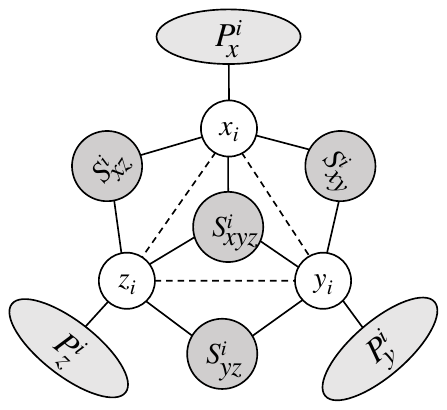}
\caption{The structure of graph $G_i$ in the proof of Theorem~\ref{theo:or_general}.}\label{fig:gadget}
 \end{figure}

Consider now the subgraph of $G_i$ induced by $x_i$, $y_i$, $z_i$, and one vertex (if any) from each set $S^i_{xy}$, $S^i_{xz}$, $S^i_{yz}$, $S^i_{xyz}$; this graph has at most 7 vertices and hence is a PCG.
In view of Property~\ref{prop:closure_pcg}.1, we can add to this graph all vertices in $P^i_x$, $P^i_y$, $P^i_z$ and all their incident edges in $G_i$, and we still get a PCG; finally, in view of Property \ref{prop:closure_pcg}.2, we add all remaining vertices in $S^i_{xy}$, $S^i_{xz}$, $S^i_{yz}$, $S^i_{xyz}$ and all incident edges in $G_i$, and we still get a PCG. Thus $G_i$ is a PCG.
\end{proof}

For particular graph classes the result in Theorem \ref{theo:or_general} can be improved.

\begin{theorem}The following statements hold:
\begin{enumerate}
\item Every connected graph with maximum degree at most $3$ is a \OR{2};
\item Every regular graph  with  even degree $\Delta$ is   a   \OR{\frac{\Delta}{2}};
\item Every bipartite regular graph  with  odd degree $\Delta$ is  a  \OR{\lceil \frac{\Delta}{2} \rceil}.
\end{enumerate}
\end{theorem}
\begin{proof} 
We prove separately the statements.
\begin{enumerate}
\item 
It is known that the edges of any connected graph with degree at most $3$ can be partitioned into a spanning forest and a subgraph whose connected components are either $K_2$s or cycles \cite{AJS15}. 
The result follows from observing that a forest is a PCG and the same holds for the disjoint union of edges and cycles. 
\item
Petersen \cite{P1891} proved that the edges of every regular graph with even degree $\Delta$ can be covered with $\frac{\Delta}{2}$ sets of vertex disjoint cycles; the result follows as a cycle is a PCG \cite{YBR10}.
\item
We exploit the Hall's marriage theorem proving that a bipartite regular graph always contains a perfect matching (that is a special forest and hence is a PCG);
by removing it from the graph, we get a regular graph with even degree $\Delta -1$ and the result in item 2 of this theorem can be used. 
\end{enumerate}
\vspace*{-.6cm}
\end{proof}

Next we focus on the class of planar graphs. In \cite{BCMP19,DMR15} it is shown that not all planar graphs are PCGs while it is not known whether they are a subclass of \Int{2}. We prove that planar graphs are in \OR{3}. 
It is hence interesting to study which subclasses of planar graphs are in \OR{2} and which superclasses of planar graphs are in  \OR{4}.

For the sake of completeness, before stating our results, we give a brief summary of the definitions of the considered classes and of the inclusion relations among them.

{\em Planar graphs} can be drawn on the plane in such a way that no edges cross each other. Equivalently, planar graphs do not contain  $K_5$ or $K_{3,3}$ as minors  \cite{Wagner1937}. (Graph $H$ is a {\em minor} of graph $G$ if $H$ can be formed from $G$ by deleting edges and vertices and by contracting edges.)

{\em Series-parallel} graphs  do not contain $K_4$ as a minor \cite{Eppstein1992}. 
Since $K_4$ is a minor of both $K_5$ and $K_{3,3}$, series-parallel graphs are planar graphs.
It has been proved that a strict subclass of series-parallel graphs ({\em i.e.} SQQ series-parallel graphs) is in \Int{2} \cite{multiintervalPCG} while it remains an open problem whether the whole class of series-parallel graphs is in \Int{2}. We prove that series-parallel graphs are in \OR{2}.

Finally, {\em 1-planar graphs} are graphs that can be drawn in the plane such that any edge intersects with at most one other edge. 
They are a superclass of planar graphs.
We prove that they are in \OR{4}.

\begin{theorem} 
\label{th.OR_planar}
The following statements hold:

\begin{enumerate}
\item Every planar graph is a \OR{3};
\item Every triangle-free planar graph  is a \OR{2};
\item Every series-parallel graph  is a \OR{2};
\item Every 1-planar graph is a \OR{4}.
%\item every outerplanar graph is \OR{2} (e' conseguenza di quella del punto $3$)
\end{enumerate}
\end{theorem}
\begin{proof} We prove the statements separately.
\begin{enumerate}
\item 
%by euler's planarity theorem any planar graph with $n$ vertices has at most 
%$3n-6$ edges, from which it follows 
From  Nash-Williams' formula \cite{N61} and Euler's planarity theorem, we have that planar graphs have arboricity at most 3. 

\item From Nash-Williams' formula and Euler's planarity theorem we have that
planar triangle free graphs have arboricity at most $2$.

\item  The statement follows by the fact that the edges of any series-parallel graph can be covered by two forests \cite{BKPY05}.
\item   Given a  1-planar graph, there is a partition of  its edges into two subsets $A$ and $B$  such that  $A$ induces a planar graph and  $B$ induces a forest \cite{A14}.
The statement follows from the fact that, in turn, the planar graph can be covered by at most 3 forests.
\end{enumerate}
The application of Lemma~\ref{lemma:arboricity} concludes the proof.
\end{proof}

We conclude this section with an observation on {\em outerplanar graphs}, {\em i.e.} planar graphs that can be drawn in the plane so that all vertices are on the same face \cite{outerplanar}.
%It is worth to note that 
It is not known whether these graphs are in PCG, but a strict subclass ({\em i.e.} triangle-free outerplanar 3-graphs) has been proved to be in PCG \cite{SRH13}.
An equivalent definition of outerplanar graphs characterizes them as graphs  that do not contain either $K_4$ or $K_{2,3}$ as minors; hence, they are a subclass of series-parallel graphs.
From our result on series-parallel graphs in Theorem \ref{th.OR_planar}, we have that outerplanar graphs are in \OR{2}.

%%%%%%%%%%%%%%%%%%%%%%%%%%%%%%%%%%%%%%%%%%%%%%%%%%%%%%%%%%%%%%%%%%%
\section{On AND-PCGs}\label{sec:and}

The {\em intersection dimension} of a graph $G$ with respect to a class of graphs $\mathcal{C}$ is the minimum $k$ such that $G$ is the intersection of some $k$ graphs belonging to $\mathcal{C}$ \cite{Kratochvil1994}.
From Definition~\ref{def:second-def-AND}, the following lemma holds: \\
\begin{lemma}
\label{lemma:intersection}
    A graph $G$ is a \AND{k}\ if and only if the intersection dimension of  $G$ 
% \cite{Kratochvil1994}
with respect to the PCG class is at most $k$. 
\end{lemma}

Thus, we can exploit the literature on the intersection dimension of graphs.  Among the different specializations of intersection dimension, the most well-known
is the \emph{boxicity} (\textit{i.e.} the intersection dimension with respect to the class of interval graphs) \cite{Roberts}. An {\em interval graph} is an undirected graph formed from a set of intervals on the real line, with a vertex for each interval and an edge between vertices whose intervals intersect \cite{Brandstdt1999}. 
Since interval graphs are known to be in PCG \cite{Brandstand2008}, 
in view of Lemma \ref{lemma:intersection}, it holds that 
\begin{corollary}\label{cor:boxicity}
A graph with boxicity $b$ is a \AND{b}.
\end{corollary}

The following lemma is based on  Definition~\ref{def:second-def-AND}.

\begin{lemma}\label{lem:and_complement}
A graph $G=(V,E)$ is a \OR{k}$(G_1,\ldots, G_k)$ with $G_i \in mLPG \cup LPG$ for all $i$, $1\leq i \leq k$ if and only if $\overline{G}$ is a \AND{k}$(\overline{G_1},\ldots, \overline{G_k})$ with $\overline{G_i} \in mLPG \cup LPG$ for all $i$, $1\leq i \leq k$. 
\end{lemma}
\begin{proof}
A graph  $G_i$ is in $ mLPG \cup LPG$ if and only if $\overline{G_i}$ is in $mLPG \cup LPG$ as the union of the classes mLPG and LPG is closed with respect to the complement \cite{CP12}. 
The following implications hold:

$$
e \in E(\overline{G}) \Longleftrightarrow e \not\in E(G) \Longleftrightarrow e \not \in \cup_{i} E(G_i) \Longleftrightarrow e \in \cap_{i} E(\overline{G_i}).
$$

Hence, $G=~$\OR{k}$(G_1,\ldots, G_k)$  if and only if  $\overline{G}=~$\AND{k}$(\overline{G_1},\ldots, \overline{G_k})$.
\end{proof}

\begin{theorem}
\label{theo:AND}
Any graph $G=(V,E)$ with $n$ vertices, minimum degree $\delta$ and maximum degree $\Delta$ is a \AND{\min \{\lfloor \frac{n}{2}\rfloor,\lceil\frac{3(n-\delta)-1}{5}\rceil,\Delta \log^{1+o(1)}\Delta\}}.
%$k=\min\{arb(\overline{G}), \}$
\end{theorem}
\begin{proof}
The result $\lfloor \frac{n}{2}\rfloor$ follows from the same value of the boxicity of any $n$ vertex graph \cite{Roberts} and from Corollary~ \ref{cor:boxicity}.

The value $\lceil\frac{3(n-\delta)-1}{5}\rceil$ follows directly by combining   Lemma \ref{lem:and_complement} with  Theorem \ref{theo:or_general} and observing that forests of paths are LPGs and that the maximum degree of $\overline{G}$ is equal to $n-1-\delta$.

The value $\Delta \log^{1+o(1)}\Delta$ comes from Corollary~ \ref{cor:boxicity} and from the result in \cite{Scott2019} that graphs with maximum degree $\Delta$ have boxicity upper bounded by $\Delta \log^{1+o(1)}\Delta$. 
\end{proof}

In the particular case of planar graphs we can improve the result of Theorem~\ref{theo:AND}:

\begin{theorem}
The following statements hold:
\begin{enumerate}
    \item Every planar graph is a \AND{3};
    \item Every outerplanar graph is a \AND{2}.
\end{enumerate}
\end{theorem}

\begin{proof}
The results can be directly derived by Corollary~ \ref{cor:boxicity} and observing that every planar graph has boxicity at most   $3$  \cite{Thomassen1986} and every outerplanar graph has boxicity at most $2$ \cite{Scheinerman1984}.
\end{proof}
%Notice that the constraint on $G$ of  Lemma~\ref{lem:and_complement} is much stronger than the simple requirement that $G$ is a $\OR{k}$. Indeed, requiring  only that $G$ is a $\OR{k}$ does not directly imply that we can say something about the value of $k'$ for which $\overline{G}$ is in $\AND{k'}$. Indeed, when $G$ is a $\OR{k}$

Notice that if the hypothesis of Lemma \ref{lem:and_complement} does not hold (\textit{i.e.} $G_i \not \in mLPG \cup LPG$) we have the following weaker result: 

\begin{theorem}\label{theo: and_complement_to_or} 
If a  graph $G$ is a \AND{k}, its complement graph $\overline{G}$ is a \OR{2k}.
\end{theorem}
\begin{proof}
First note that, any graph $H=PCG(T, I)$ with $I=[a,b]$ can be expressed as a \AND{2}($H^1, H^2$) where $H^1=PCG(T,I^1)$ with $I^1=[0, b]$ and $H^2=PCG(T, I^2)$ with $I^2=[a, \mu(T)]$, hence $H^1, H^2$ are both in $mLPG \cup LPG$.

Exploiting this fact, from $G=~$\AND{k}$(G_1, \ldots , G_k)$ we deduce that we also have $G=~$\AND{2k}$(G_1^1$, $G_1^2$, $\ldots$ , $G_k^1$, $G_k^2)$ where $G_i^1$, $G_i^2$ are in $mLPG \cup LPG$.
Finally, from Lemma \ref{lem:and_complement}, it follows that $\overline{G}=~$\OR{2k}$(\overline{G_1}^1, \overline{G_1}^2, \ldots , \overline{G_k}^1, \overline{G_k}^2)$.
%Let $(T_i,I_i)$ for $i\in [k]$ be the construction for $k$-AND-PCG. The construction for $2k$-OR-PCG comes as follows: for each $(T_i,I_i)$ with $I_i=[a_i,b_i]$, we create two pairs $(T_i,I'_i)$ and $(T_i,I''_i)$ where $I'_i=[0,a_i-1]$ and $I''_i=[b_i+1, MAX]$ where max is the maximum pairwise distances between the leaves of $T_i$.
\end{proof}

Trivially, $PCG\subseteq~\OR{2}$; it is well known that there are graphs in \Int{2}  but not in PCG \cite{multiintervalPCG}, so implying that
$PCG\subsetneq~\OR{2}$.
Now we will show that
 $PCG\subsetneq \,\,\AND{2}.$

\begin{theorem}
$PCG\subsetneq \,\,\AND{2}\,\,  \cap \,\, \Int{2} $.
\end{theorem}
\begin{proof}
Trivially $PCG \subseteq$ \AND{2} and $PCG \subseteq$ \Int{2}.

Consider  the regular graph $G$ in Figure~\ref{y}(a), that is not in PCG \cite{DMR15} but is in \Int{2} \cite{multiintervalPCG}. 
To conclude the proof, we show that $G$ is in \AND{2} $\cap$ \Int{2}. 
We show this fact in two different ways: in the first one the graphs involved in the intersection are interval graphs, while in the second one they are not.

\smallskip

Applying Theorem~\ref{theo:AND} with $n=8$ and  $\delta=5$ we immediately obtain that $G$ is a \AND{2} and, as its proof is conceived,  the graphs involved in the intersection are interval graphs.

\smallskip

Alternatively, consider the graph $G_1$ depicted in Figure~\ref{x}(a); $G_1$ is a PCG as it has 7 vertices.
We add to $G_1$ a vertex $v_6$ with the same neighborhood as $v_4$, so obtaining graph $G'_1$ in Figure~\ref{x}(b) that is also a PCG in view of Property \ref{prop:closure_pcg}.2. Let $G'_1=PCG(T_1, I_1)$.

Analogously, consider the PCG $G_2$ in Figure~\ref{x}(c).
We add to $G_2$ a vertex $v_7$ with the same neighborhood as $v_3$, so obtaining PCG $G'_2$ in Figure~\ref{x}(d).
Let $G'_2=PCG(T_2, I_2)$.

Consider now \AND{2}$(T_1, T_2,I_1, I_2)$; this graph is exactly $G$ (differences between $G'_1$ and $G'_2$ are dotted in Figures~\ref{x}(b) and \ref{x}(d)). Finally, notice that neither $G'_1$ nor $G'_2$ is an interval graph as interval graphs are chordal whereas both $G'_1$ and $G'_2$ contain a chordless cycle (see for example the cycles $v_1,v_2,v_7,v_8$ in $G'_1$ and $G'_2$).
\end{proof}

\begin{figure}[h]
\centering
    \includegraphics[scale=0.6]{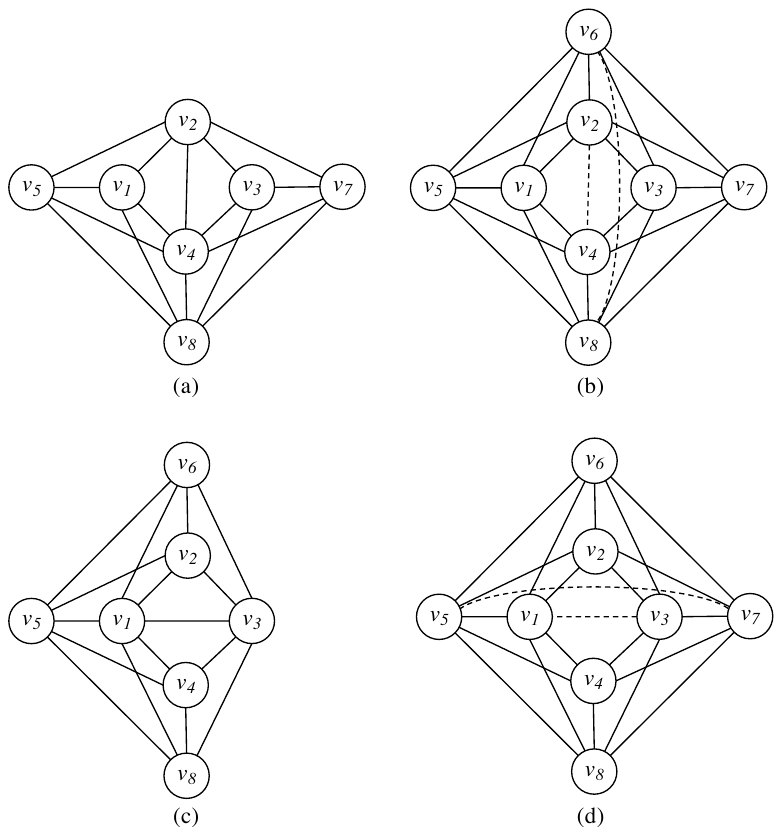}
    \caption{(a) Graph $G_1$ (b), Graph $G'_1$ (c), Graph $G_2$, (d) Graph $G'_2$}\label{x}
  \end{figure}

Besides the graph in Figure~\ref{y}(a), it is possible to find many other  graphs that are not PCGs but are in \AND{2}, as shown in the following.

We say that a graph $G$ is {\em maximal non-PCG} if $G$ is not PCG but the addition of any  edge produces a PCG. Notice that these graphs exist: for any graph $G$ that is not a PCG, either $G$ is maximal non-PCG, and we are done, or there exists an edge we could add and such that the graph remains outside the PCG class; iterating this process and noticing that the complete graph is a PCG, we eventually reach a maximal non-PCG.

\begin{theorem}
Any maximal non-PCG is a  \AND{2}.
\end{theorem}
\begin{proof}
Preliminarily, observe that a maximal non-PCG $G$ has at least 2 non-edges. 
In fact, a complete graph without a single edge is a PCG because a graph whose complement is acyclic is a PCG \cite{Hossain_2017}.
So, let $e_1, e_2 \not \in E(G)$. 
As $G$ is a maximal non-PCG, the 
graphs 
%$G_1=G \cup \{ e_1 \}$ and $G_2=G \cup \{ e_2 \}$ 
$G_1=(V(G), E(G) \cup \{ e_1 \})$ and $G_2=(V(G), E(G) \cup \{ e_2 \})$ are two PCGs and their intersection is $G$.
\end{proof}

Notice that from \cite{BCMP19} a graph is a \emph{minimal non-PCG} if it is not PCG and removing any edge gives a PCG.  It is worth noting that there are graphs that are both maximal and minimal non-PCGs and the graph $G$ in Figure~\ref{y}(a) provides such an example. Indeed, in \cite{Azam2021} it is showed that there are exactly seven graphs with eight vertices that are not PCGs and  their list is provided. It is not difficult to see that any graph $G'$ obtained by removing or adding an arbitrary edge from $G$, is not in this list, and thus $G'$ is a PCG. Hence, we deduce that $G$ is both minimal and maximal.

We now show an example of a graph that is not in \AND{2}.  

\begin{lemma}\label{lem:k4-4}
Let $H$ be the graph consisting of three disjoint copies of $K_{4,4}$. If the edges of $H$ are colored with two colors then there exist  two disjoint monochromatic chordless cycles of the same color. Moreover, these cycles belong to different  copies of $K_{4,4}$.
\end{lemma}
\begin{proof}
Let the edges of $H$ be colored in some way with colors green and blue.

We show first that if the edges of $K_{4,4}$ are colored with two colors then necessarily we have a monochromatic chordless cycle. Indeed, the two monochormatic subgraphs induced by each one of the two colors cannot be both forests as each forest on 8 vertices has at most 7 edges and thus two forests would cover only 14 edges instead of 16. Thus, without loss of generality, we assume that the red color induces a subgraph that contains at least one cycle and let $C$ be the smallest one. As we are considering subgraphs of a bipartite graph, $C$ cannot have length $3$, thus it has length at least $4$. Moreover, $C$ is chordless as if there was a chord we could find a cycle of length strictly smaller, which contradicts the hypothesis that $C$ is a cycle of smallest length.

Finally, the proof of the claim follows as $H$ contains 3 disjoint copies of $K_{4,4}$ and each one of them contains a monochromatic chordless cycle and thus two of these cycles  will necessarily  have the same color. 
\end{proof}

%In the next theorem we use the graph $G=\overline{H}$.

\begin{theorem}\label{theo:not_and}
Graph $\overline{H}$ is not a \AND{2}.
\end{theorem}
\begin{proof}
Suppose on the contrary that $\overline{H}$ is a \AND{2}. 
From Definition~\ref{def:second-def-AND}, there exist $2$ graphs $G_1, G_2$ on the same vertex set such that $G_1 \in PCG$, $G_2 \in PCG$ and $E(\overline{H})= E(G_1) \cap E(G_2)$. 
We denote $E'= E(G_1) \setminus E(\overline{H})$ and $E''= E(G_2) \setminus E(\overline{H})$ with $E'$ and $E''$ possibly empty. 
Set $E'$ ($E''$) contains the edges we added to $\overline{H}$ in order to obtain a PCG $G_1$ ($G_2$). Clearly we  have $E' \cap E'' = \emptyset$, $E' \subseteq E(H)$ and $E'' \subseteq E(H)$. 

We can view the edges of $H$ as colored by three colors: the edges in $E'$ colored by red, the edges in $E''$ colored by blue and the edges of $H$ that are neither in $E'$ nor in $E''$  colored by black. 

We proceed to show that at least one between $G_1$ and $G_2$ is not a PCG, contradicting our initial hypothesis. 
To this purpose we use the result in \cite{Hossain_2017} which states that any graph whose complement has two disjoint chordless cycles (\textit{i.e.} two cycles that are vertex disjoint and for which there exists no edge connecting two vertices that belong to different cycles), is not a PCG.  
Notice that the complement of graph $G_1$ is exactly the subgraph of $H$ induced by the blue and black edges  and the complement of graph $G_2$ is the one induced by red and black edges. 
Thus, it is sufficient to prove that either the red-black subgraph or the blue-black subgraph contains two disjoint chordless cycles. 
We consider the red-black induced subgraph ({\em i.e.}, the complement of graph $G_2$) as colored by green. 
Thus we obtain a coloring of the edges of $H$ with two colors, green and blue, and we can use Lemma~\ref{lem:k4-4}, having that at least one of the followings hold:

(a) the green subgraph contains two disjoint  chordless cycles;

(b) the blue subgraph  contains two disjoint  chordless cycles.

\noindent
If case (a) holds we are done as the green subgraph corresponds to $\overline{G_2}$, implying the contradiction that $G_2$ is not a PCG. 

If case (b) holds, then  by adding the black edges to the blue subgraph (in order to obtain $\overline{G_1}$) we have two disjoint chordless cycles. 
Indeed, a chordless blue cycle plus some black edges will always contain, a chordless cycle of at least $4$ vertices as we are considering subgraphs of $H$ that is a bipartite graph.  
Moreover, the black edges cannot connect the two disjoint blue chordless cycles  as, from Lemma~\ref{lem:k4-4}, each one of them belongs to a different connected component of $H$. 
We then deduce that $G_1$ is not a PCG. In conclusion, $G$ is not a \AND{2}.
\end{proof}

\subsection*{Not all graphs are OR-PCGs or AND-PCGs}\label{subsec:notOrnotAnd}
Given the above results, it is natural to ask whether there is a constant $t$ such that all graphs are $t$-OR or $t$-AND PCGs.  We show that there is no such $t$ and that, in fact, we can strengthen this result by showing that not all graphs can be obtained from the OR (or AND) of $t$ \emph{$k$-interval} PCGs.
We say that a graph $G$ is a $t$-OR $k$-interval PCG if there exists a set of $k$-interval PCGs $G_1, \ldots, G_t$ such that $V(G) = V(G_1) = \ldots = V(G_t)$ and $E(G) = \cup_{i=1}^t E(G_i)$.  
Similarly, $G$ is a $t$-AND $k$-interval PCG if instead $E(G) = \cap_{i=1}^t E(G_i)$.  

It is not hard to show that a graph $G$ is \emph{not} a $t$-OR $k$-interval PCG if and only if, for every way of assigning each $e \in E(G)$ to a non-empty subset of colors in $\{1, \ldots, t\}$, there is some $c \in \{1, \ldots, t\}$ such that the edges that have color $c$ form a subgraph that is not a $k$-interval PCG (to see this, interpret giving color $c$ to $e \in E(G)$ as having $e \in E(G_c)$).
This shows that finding non $t$-OR $k$-interval PCGs is a Ramsey-type question where edges are colored by subsets, and where we want one of the colors to be outside of a specified graph class.  Even with this observation, it is not immediate that non $t$-OR $k$-interval PCGs exist.  We combine Ramsey theory with Theorem~\ref{theo:not-interval} for this purpose.

For a given bipartite graph $G$, we write $R \rightarrow_{\epsilon~ind} G$ and say that a bipartite graph $R$ has the \emph{induced $\epsilon$-density property} if every subgraph of $R$ with at least $\epsilon |E(R)|$ edges contains a copy of $G$ which is an induced subgraph of $R$.  That is, $G$ is an  induced subgraph of both the dense subgraph and $R$ itself.   
The following was shown in \cite[Theorem 1.10]{dudek2013some} (the notation $c = c(\epsilon)$ is used to indicate that the appropriate value of $c$ is a function of $\epsilon$, which is shown to exist but is not determined exactly by the authors).

\medskip 

\begin{theorem}[\cite{dudek2013some}]\label{thm:ramsey-induced}
For every $\epsilon$, $0 < \epsilon < 1$, there is a constant $c = c(\epsilon)$ such that for each $n$, there exists a bipartite graph $R$ with at most $2^{cn}$ vertices such that $R \rightarrow_{\epsilon~ind}
G$ for every bipartite graph $G = (V_1 \cup V_2, E)$ with $|V_1|, |V_2| \leq n$.
\end{theorem}

The above means that, in particular, any dense enough subgraph of $R$ contains every bipartite graph, up to a certain size, as an \emph{induced} subgraph (these also happen to be induced subgraphs of $R$, but we will not use this fact).
The proof of Theorem~\ref{thm:ramsey-induced} is constructive, and results in the bipartite graph $R = (V_1 \cup V_2, E)$ in which, putting $N = cn$ for a large enough $c = c(\epsilon)$,  we have $V_1 = N$, $V_2 = {N \choose {N/2}}$, and every vertex of $V_2$ is connected to all vertices of precisely one $N/2$-subset of $V_1$.  
Interestingly, the $R$ graph with $N = 4k-2$ is almost identical to the $H_k$ graph from Theorem~\ref{theo:not-interval}.

We will make use of the following.

\begin{lemma}\label{lem:and-to-or}
Let $G$ be a $t$-AND $k$-interval PCG.  Then $\overline{G}$ is a $t$-OR $(k + 1)$-interval PCG.
\end{lemma}

\begin{proof}
Since $G$ is a $t$-AND $k$-interval PCG, there are graphs $G_1, \ldots, G_t$ such that 
$E(G) = E(G_1) \cap \ldots \cap E(G_t)$.
Let $i \in \{1, \ldots, t\}$.  Then $G_i$ is a $k$-interval PCG and we can write $G_i = k\textnormal{-}PCG(T_i, I_{i, 1}, \ldots, I_{i, k})$, where $T_i$ is a tree and the $I_{i, j}$'s are intervals.  We denote $\I_i = I_{i, 1} \cup \ldots \cup I_{i, k}$.
Note that the set $\J_i = \mathbb{N} \setminus \I_i$ can be partitioned into $k + 1$ intervals, which we will denote $\{J_{i, 1}, \ldots, J_{i, k + 1}\}$. 
Define $G'_i = (k + 1)\textnormal{-}PCG(T_i, J_{i, 1}, \ldots, J_{i, k+1})$.  

Consider the graph $G'$ obtained from the OR of the $G'_i$ graphs, \textit{i.e.} $V(G') = V(G)$ and $E(G') = E(G'_1) \cup \ldots \cup E(G'_t)$.
We claim that $G' = \overline{G}$.
Let $\{u,v\} \in E(G)$.  Then, since $G$ is obtained from the AND of the $G_i$ graphs, for each $i \in \{1, \ldots, k\}$ we have $d_{T_i}(u, v) \in \I_i$, and thus $d_{T_i}(u, v) \notin \J_i$.  This implies that $\{u,v\} \notin E(G')$.
Let $\{u,v\} \notin E(G)$.  Then there exists $i \in \{1, \ldots, t\}$ such that $d_{T_i}(u, v) \notin \I_i$, and thus $d_{T_i}(u, v) \in \J_i$ and therefore $\{u,v\} \in E(G')$.  It follows that $G'$ is indeed the complement of $G$.
\end{proof}

\begin{theorem}
Let $t \geq 2, k \geq 1$ be integers.
Then there exist graphs that are not $t$-OR $k$-interval PCGs, and there exist graphs that are not $t$-AND $k$-interval PGCs.
\end{theorem}

\begin{proof}
Let us begin with $t$-OR $k$-interval PCGs.
As shown in~Theorem~\ref{theo:not-interval}, for any $k$ there exists a bipartite graph $H_k = (V_1 \cup V_2, E)$ with $|V_1|, |V_2| \leq n$ for some large enough $n$, such that $H_k$ is not a $k$-interval PCG.
Put $\epsilon = 1/t$.
By Theorem~\ref{thm:ramsey-induced}, there is a constant $c = c(\epsilon)$ and a bipartite graph $R$ of order $2^{cn}$ such that $R \rightarrow_{\epsilon~ind} H_k$.  

Towards a contradiction, assume that $R$ is a $t$-OR $k$-interval PCG, obtained by taking the union of the edges of $k$-interval PCGs $G_1, \ldots, G_t$.  
Because each edge of $R$ is in at least one $G_i$ graph, there is $i \in \{1, \ldots, t\}$ such that $|E(G_i)| \geq |E(R)|/t = \epsilon |E(R)|$.  Since $R \rightarrow_{\epsilon~ind} H_k$, $G_i$ contains a copy of $H_k$ as an induced subgraph.  Therefore, $G_i$ is not a $k$-interval PCG, a contradiction.  This shows that $R$ cannot be a $t$-OR $k$-interval PCG.

We now consider $t$-AND $k$-interval PCGs.
As we just argued, there exists a graph $R$ that is not a $t$-OR $(k + 1)$-interval PCG.  By taking the contrapositive of Lemma~\ref{lem:and-to-or}, we have that $\overline{R}$ is not a $t$-AND $k$-interval PCG.
\end{proof}

\section{Discussion and open problems}\label{sec:conclusion}
In this paper we defined two new classes of graphs: \OR{k}s and \AND{k}s, that are two different generalizations of PCGs.
We studied these classes and another already known generalization of PCGs, {\em i.e.}, \Int{k}s.
In particular, we showed that there is no constant $k$ for which the \Int{k} class contains all graphs; we provided upper bounds on the minimum $k$ for which arbitrary and special graphs belong to \OR{k} and to \AND{k} classes.
Finally, we showed that for any $k$, there exist graphs that are not \OR{k}s, and graphs that are not \AND{k}s. This work leads to numerous further challenging problems. We detail some of them below.

First, from a computational complexity point of view, nothing is known concerning \Intg s and our new graph classes. In fact,  it is not even known the computational complexity of deciding whether a graph $G$ is a PCG or not \cite{DMR15,Rahman2020}.

\begin{problem} Given a graph $G$ and an integer $k$, determine the exact complexity of deciding whether $G$ is in \OR{k}\ (or in \AND{k} or in \Int{k}). 
In particular, if $k$ is part of the input, is it NP-hard to decide membership in one of these classes?
 \end{problem}

From Lemma~\ref{lem:complement_interval}, if $G$ is a PCG then its complement $\overline{G}$ is a \Int{2} and thus a \OR{2}. The analogous question related to the \AND{2}\ class remains an open problem.
\begin{problem}
If $G$ is a PCG determine whether  $\overline{G}$ is  a \AND{2}.
\end{problem}

In Figure~\ref{y}.(a) we show a graph $G$ that is not a PCG. Notice that its complement, $\overline{G}$, consists of two disjoint copies of $C_4$ and thus  $\overline{G}$ is a PCG. Hence, the PCG class is not closed with respect to the complement. However, it is not known whether the same holds for the other generalizations of the PCGs. 

\begin{problem}
Determine whether the classes \Intg, \ORg\ and \ANDg\ are closed  with respect to the complement. 
%Can we have a result similar to Lemma~\ref{lem:complement_interval} for \ORg\ (\ANDg)? More specifically, given a graph $G$ in \OR{k}\ (\AND{k}), what is the minimum $k'$ such that $\overline{G}$ is in \OR{k'} (\AND{k'})?
\end{problem}
\noindent
In relation to the last problem with respect to the \Intg\ class, notice that the result of Lemma~\ref{lem:complement_interval} indicates only that if $G$ is a \Int{k}\ then $\overline{G}$ is a \Int{(k+1)}.

From Theorem~\ref{theo: and_complement_to_or} we know that if a graph $G$ is a \AND{k}\, its complement $\overline{G}$ is a \OR{2k}. The reverse case seems more difficult.

\begin{problem}
If $G$ is a \OR{k}\, what can we say about the value of $k'$ for which  $\overline{G}$ is a \AND{k'}?
\end{problem}

It is known that the smallest graph that is not a PCG has 8 vertices \cite{Calamoneri2012,DMR15}. 
From Theorem~\ref{theo:not-interval}, we have graph $H_3$ of 135 vertices that is not a \Int{2}. 
Similarly, from Theorem~\ref{theo:not_and}, we have a graph of 24 vertices that is not a \AND{2}. 

\begin{problem}
What is the smallest value of $n$ such that there exists an $n$ vertex
graph that is not a \Int{2} (\AND{2}, \OR{2})?
\end{problem}

Finally, we show that graph $H_k$ in Theorem~\ref{theo:not-interval} is not in \Int{(k-1)}, the following question remains open.
\begin{problem}
Determine the minimum $t \geq k$ such that $H_k$ belongs to \Int{t}.
\end{problem}

%\section*{Acknowledgments}
%T. Calamoneri and A. Monti were partially supported by Sapienza University of Rome, grants RM122181612C08BB and RM123188D7F7985D. 
%B. Sinaimeri  was supported by MUR PRIN Project EXPAND, grant number 2022TS4Y3N.
\nocite{*}
\bibliographystyle{abbrvnat}
% use the following instead if you encounter problems 
%\bibliographystyle{alpha}
\bibliography{reference}
\label{sec:biblio}

\end{document}